\documentclass[11pt,a4paper]{article}
\usepackage{amscd}
\usepackage{enumerate}

\usepackage{amsmath, amssymb, latexsym}
\usepackage{bbm}
\usepackage{amsfonts}
\usepackage{amsthm}
\usepackage{relsize}
\usepackage{setspace}
\usepackage{geometry}
\usepackage{url}
\usepackage{xspace}
\usepackage{tocloft}
\usepackage{graphics}
\usepackage{graphicx}
\usepackage{lscape}
\usepackage{microtype}
\usepackage{ulem}

\usepackage[usenames, dvipsnames]{color}
\usepackage[utf8]{inputenc}
\usepackage{tikz}

\usepackage[pagebackref=true]{hyperref}
\usepackage[alphabetic]{amsrefs}
\usepackage[english]{babel}

\usepackage{authblk}

\newtheorem{theorem}{Theorem}[section]
\newtheorem{proposition}[theorem]{Proposition}

\newtheorem{corollary}[theorem]{Corollary}
\newtheorem{lemma}[theorem]{Lemma}

\theoremstyle{definition}
\newtheorem{definition}[theorem]{Definition}

\newtheorem{remark}[theorem]{Remark}
\newtheorem{notation}[theorem]{Notation}
\newtheorem{convention}[theorem]{Convention}

\theoremstyle{problem}

\newcommand{\Aut}{\mathrm{Aut}}

\newcommand{\CC}{\mathbf{C}}

\newcommand{\Min}{\mathrm{Min}}

\newcommand{\bd}{\partial}

\newcommand{\id}{\operatorname{id}}

\newcommand{\End}{\operatorname{End}}

\newcommand{\T}{\operatorname{\mathcal{T}}}
\newcommand{\G}{\operatorname{\mathbb{G}}}
\newcommand{\dist}{\operatorname{dist}}

\newcommand{\Sym}{\operatorname{Sym}}

\newcommand{\kk}{\mathrm{\bold{k}}}

\def\og{\leavevmode\raise.3ex\hbox{$\scriptscriptstyle\langle\!\langle$~}}
\def\fg{\leavevmode\raise.3ex\hbox{~$\!\scriptscriptstyle\,\rangle\!\rangle$}}

\title{Infinitely generated Hecke algebras with infinite presentation}

\author[1]{Corina Ciobotaru\thanks{corina.ciobotaru@unifr.ch}}

\date{March 16, 2016}

\begin{document}
\newcounter{qcounter}

\maketitle

\begin{abstract}
For a locally compact group $G$ and a compact subgroup $K$, the corresponding Hecke algebra consists of all continuous compactly supported complex functions on $G$ that are $K$--bi-invariant. There are many examples of totally disconnected locally compact groups whose Hecke algebras with respect to the maximal compact subgroups are not commutative. One of those is the universal group $U(F)^+$, when $F$ is primitive but not $2$--transitive. For this class of groups we prove that the Hecke algebra with respect to the maximal compact subgroup $K$ is infinitely generated and infinitely presented. This may be relevant for constructing irreducible unitary representations of $U(F)^+$ whose subspace of $K$--fixed vectors has dimension at least two, or answering the question whether $U(F)^+$ is a type I group or not. On the contrary, when $F$ is $2$--transitive that Hecke algebra of $U(F)^+$ is commutative, finitely generated admitting only one generator.

\end{abstract}

\section{Introduction}
\label{sec::introduction}

The Hecke algebras are very useful tools to study the representation theory of locally compact groups. For example, in the particular case of a semi-simple algebraic group $G$ over a non-Archimedean local field there are two important Hecke algebras that can be associated with: the Hecke algebra of $G$ with respect to the good maximal compact subgroup, called the spherical Hecke algebra of $G$, and the Hecke algebra of $G$ with respect to the Iwahori subgroup, which is a smaller compact subgroup. The latter algebra is called the Iwahori--Hecke algebra of $G$ and plays a very important role in the representation theory of algebraic groups, especially in the Kazhdan--Lusztig theory, being an intense and rich field of research. The former one is used to study the spherical unitary dual of semi-simple and analogous groups. That Hecke algebra is moreover commutative and finitely generated with respect to the convolution product. The representation theory of both algebras are intimately related to the representation theory of $G$. 

In this article we restrict our attention to Hecke algebras associated with specific totally disconnected locally compact groups and their maximal compact subgroups. First, let us recall the general definition of a Hecke algebra and some more specific motivation.
\begin{definition}
\label{def::hecke_algebra}
Let $G$ be a locally compact group and $K \leq G$ be a compact subgroup. We denote by $C_{c}(G, K)$ the space of continuous, compactly supported  complex-valued functions $\phi : G \to \CC$ that are  \textbf{$K$-bi-invariant}, i.e., functions that satisfy the equality $\phi(kgk{'})=\phi(g)$  for every  $ g \in G $ and  all $ k, k{'} \in K$.
We view the $\CC$-vector space $C_{c}(G, K)$ as an algebra whose multiplication is given by the convolution product
\[
 \phi \ast \psi \colon x \mapsto  \int_{G} \phi(xg )\psi(g^{-1})d\mu(g)
\]
where $\mu$ is the left Haar measure on $G$.
Moreover, $C_{c}(G,K)$ is called the \textbf{Hecke algebra} corresponding to $K \leq  G$ and $(G,K)$ forms  a \textbf{Gelfand pair} if the convolution algebra $C_{c}(G,K)$ is commutative.
\end{definition} 

For a locally compact group $G$ it is well known that there is a natural one-to-one correspondence between (irreducible) unitary representations of $G$ and (irreducible) non-degenerate $*$--representations of $L^{1}(G)$ (see for example Dixmier~\cite[Chapter~13]{Di77}). When restricted to a Hecke algebra of $G$ we have the following situation. To any (irreducible) unitary representation $(\pi, \mathcal{H})$ of $G$ admitting a non-zero $K$--invariant vector there is associated a canonical (irreducible) non-degenerate $*$--representation of the Hecke algebra $C_{c}(G,K)$ to $\End(\mathcal{H}^{K})$, where $\mathcal{H}^{K}$ is the space of $K$--invariant vectors with respect to $(\pi, \mathcal{H})$. In general the converse is not true, namely: the category of  non-degenerate $*$--representation of the Hecke algebra $C_{c}(G,K)$ is not necessarily equivalent to the category of unitary representations of $G$ generated by $K$--fixed vectors (see for example the PhD thesis of Hall~\cite{Ha99}). 

\medskip
The importance of Gelfand pairs in the theory of unitary representations of locally compact groups is given by the following well known two results. For the corresponding definitions one can consult van Dijk~\cite{vD09}.

\begin{proposition}[See Proposition~6.3.1 in~\cite{vD09}]
\label{prop::one_dim_K}
Let $G$ be a locally compact group and $K$ a compact subgroup of it. The pair $(G,K)$ is Gelfand if and only if for every irreducible unitary representation $\pi$ of G on a Hilbert space $\mathcal{H}$ the dimension of the subspace of $K$--fixed vectors is at most one.
\end{proposition}

\begin{corollary}[See Corollary~6.3.3 in~\cite{vD09}]
\label{cor::sph_functions}
Let $(G,K)$ be a Gelfand pair. The positive-definite spherical functions on $G$ correspond one-to-one to the equivalence classes of irreducible unitary representations of G having a one-dimensional $K$-fixed vector space.
\end{corollary}

All known examples of connected and totally disconnected non-compact locally compact groups that admit Gelfand pairs are: 
\begin{enumerate}
\item 
all semi-simple non-compact real Lie groups, with finite center, together with their maximal compact subgroups
\item
\label{exam::gelfand_pairs_tot_disc}
locally compact groups $G$ that act continuously, properly, by type-preserving automorphisms and strongly transitively on a locally finite thick Euclidean building $\Delta$ together with the stabilizer in $G$ of a special vertex of $\Delta$.
\end{enumerate} 
In the latter case, by the main theorem of Caprace--Ciobotaru~\cite{CaCi} those are all Gelfand pairs that can appear. Recall that the second family of groups includes all semi-simple algebraic groups over non-Archimedean local fields and closed subgroups of $\Aut(T)$, that act $2$--transitively on the boundary of $T$, where $T$ is a bi-regular tree with valence at least $3$ at every vertex. Using the polar decomposition of those groups (see for example~\cite[Remark~4.6, Lemma~4.7]{Ci_c}) and the Lemma of Bernstein~\cite{bernstein74-type-I}, we have that the commutative Hecke algebra $C_{c}(G,K)$, where $K$ is the stabilizer in $G$ of a special vertex of its corresponding Bruhat--Tits building, is finitely generated. Moreover, the main result of Bernstein~\cite{bernstein74-type-I} states that all reductive $p$-adic Lie groups are groups of type I (see Definition~5.4.2 in~\cite{Di77}). For the case of closed non-compact subgroups of $\Aut(T_d)$ that act $2$--transitively on the boundary of $T_d$, the type I property is still open in this generality (see~\cite{demir04},~\cite{Ci_d}).

\medskip
Apart from the case of Gelfand pairs, the structure of Hecke algebras with respect to maximal compact subgroups is in general much less studied, even if those algebras can provide very useful information about the unitary representations of that locally compact groups. In this article we propose to study the structure of non-commutative Hecke algebras that are associated with a particular family of totally disconnected locally compact groups and their maximal compact subgroups. This family of groups is given by the universal groups $U(F)$ introduced by Burger--Mozes in \cite[Section~3]{BM00a}. In his PhD thesis \cite{Amann}, Amann studies these groups from the point of view of their unitary representations.

First, let us recall the basic definitions. 
\begin{definition}
\label{def::legal_color}
Denote by \textbf{$\mathbf{\T}$ the $d$-regular tree, with $d \geq 3$,} and by $\Aut(\T)$ its full group of automorphisms, endowed with the compact-open topology. Let $\iota: E(\T) \to \{1,...,d\} $ be a function from the set $E(\T)$ of unoriented edges of the tree $\T$ such that its restriction to the star  $E(x)$ of every vertex $x \in \T$ is in bijection with $\{1,...,d\} $. A function $\iota$ with those properties is called a \textbf{legal coloring} of the tree $\T$.
\end{definition}

\begin{definition}
\label{def::universal_group}
Let $F$ be a subgroup of permutations of the set $\{1,...,d\}$ and let $\iota$ be a legal coloring of $\T$. We define the \textbf{universal group}, with respect to $F$ and $\iota$, to be 
$$ U(F):= \{g \in \Aut(\T) \; \vert \; \iota \circ g \circ (\iota \vert_{E(x)})^{-1} \in F, \text{ for every x}\in \T \}. $$

By $U(F)^{+}$ we denote the subgroup generated by the edge-stabilizing elements of $U(F)$. Moreover,  Proposition~52 of \cite{Amann} tells us that the groups $U(F),U(F)^{+} $ are independent of the legal coloring $\iota$ of $\T$.
\end{definition}

From the definition we deduce that $U(F)$ and $U(F)^{+}$ are closed subgroups of $\Aut(\T)$. When $F$ is the full permutation group $\Sym(\{1,...,d\})$ we have that $U(F)= \Aut(\T)$. If $F=\id$ then $U(F)$ is the group of all automorphisms preserving the coloring $\iota$ of $\T$. 

Another key property which is used in the sequel is the following.
 
\begin{definition}[See \cite{Ti70}]
\label{def::Tits_prop}
Let $T$ be a locally finite tree and let $G \leq \Aut(T)$ be a closed subgroup. We say that $G$ has \textbf{Tits' independence property} if for every edge $e$ of $T$ we have the equality $G_e=G_{T_1}G_{T_2}$, where $T_i$ are the two infinite half sub-trees of $T$ emanating from the edge $e$ and $G_{T_i}$ is the pointwise stabilizer of the half-tree $T_i$.
\end{definition}

Following \cite{BM00a, Amann} we know that $U(F)$ and $U(F)^{+}$ have Tits' independence property.

To avoid heavy notation, for the rest of the article we refer to the following convention.

\begin{convention}
\label{not::general_notation}
Fix $d \geq 3$. We denote by $\dist_{\T}$ the usual distance on $\T$. Let $F$ be a primitive subgroup of $\Sym\{1, \cdots, d\}$. Consider fixed a coloring $\iota$ of $\T$, a vertex $x$ of $\T$ and an edge $e$ in the star of $x$. For simplicity set $\G:= U(F)^+$ and $K:=U(F)^{+}_{x}$, which is the stabilizer in $\G$ of the vertex $x$. Moreover, for a finite subset $A\subset \T$ we denote by $\G_{A}$ the pointwise stabilizer in $\G$ of the set $A$. For $\xi \in \partial \T$, $\G_{\xi}$ denotes the stabilizer in $\G$ of the ideal point $\xi$. Let $S(x, r):= \{y \in \T \; \vert \; \dist_{\T}(x,y)=r\}$, where $r \in \mathbb{N}^{*}$. Let $\T_{x,e}$ be the half-tree of $\T$ that emanates from the vertex $x$ and that contains the edge $e$. Set $V_{x,r}:= S(x, r) \cap \T_{x,e}$, for every $r \in \mathbb{N}^{*}$. We have that $\vert V_{x,r} \vert =(d-1)^{r-1}$. For every two points $y, z \in \T \cup \bd \T$, we denote by $[y,z]$ the unique geodesic between $y$ and $z$ in $\T \cup \bd \T$. For a hyperbolic element $\gamma$ in $\G$ let $\vert \gamma \vert:= \min_{x \in \T}\{ \dist_{\T}(x, \gamma(x))\}$, which is called the \textbf{translation length} of $\gamma$. Set $\Min(\gamma):=\{x \in \T \; \vert \;  \dist_{\T}(x, \gamma(x))=\vert \gamma \vert\}$.
\end{convention}

By~\cite{BM00a, Amann} the group $\G$ act $2$--transitively on the boundary $\partial \T$ if and only if $F$ is $2$--transitive. From this fact together with the main theorem in~\cite{CaCi}, applied to our case of $d$-regular trees, we obtain that $(\G, K)$ is a Gelfand pair if and only if $F$ is $2$--transitive. Therefore, when $F$ is primitive but not $2$--transitive the Hecke algebra $C_{c}(\G,K)$ is not commutative. In particular, by Proposition~\ref{prop::one_dim_K} the group $\G$ admits an irreducible unitary representation whose subspace of $K$--fixed vectors has dimension at least two. Moreover, the theory of unitary representations of $\G$ is not at all developed when $F$ is primitive but not $2$--transitive. It is therefore natural to study the structure of the Hecke algebra $C_{c}(\G,K)$ when $F$ is primitive but not $2$--transitive and its irreducible non-degenerate $*$--representations. Regarding the structure of this algebra we obtain the following theorem, which is the main result of this article.

\begin{theorem}
\label{thm::inf_gen_hecke_alg_int}
Let $F$ be primitive. If $F$ is $2$-transitive then the Hecke algebra $C_{c}(\G,K)$ is finitely generated admitting only one generator. If $F$ is not $2$--transitive, then the Hecke algebra $C_{c}(\G,K)$ is infinitely generated with an infinite presentation.
\end{theorem}

Understanding irreducible non-degenerate $*$--representations of $C_{c}(\G,K)$ when $F$ is primitive but not $2$--transitive is left for a further study. It would be also interesting to see if Theorem~\ref{thm::inf_gen_hecke_alg_int} could be used to answer the open question whether $\G$ is not a type I group when $F$ is primitive but not $2$--transitive and to construct an explicit irreducible unitary representation of $\G$ whose subspace of $K$--fixed vectors has dimension at least two.

The article is structured as follows. In Section~\ref{sec::combinaorial_formulas} we prove combinatorial formulas that are essential for the proof of our main theorem. In Section~\ref{sec::Hecke_alg} we study the structure of the Hecke algebra $C_{c}(\G,K)$ when $F$ is primitive. This is used in Section~\ref{sec::infinite_gen} to prove Theorem~\ref{thm::inf_gen_hecke_alg_int}.

\subsection*{Acknowledgement} We would like to thank Anders Karlsson for his encouragement to continue working on this article and his useful comments and Laurent Bartholdi, Ofir David, Iain Gordon, Alex Lubotzky and Agata Smoktunowicz for further discussions regarding finitely and infinitely generated algebras. This article was written when the author was a post-doc at the University of Geneva and University of M\"{u}nster. We would like to thank those institutions for their hospitality and good conditions of working.

\section{Some combinatorial formulas}
\label{sec::combinaorial_formulas}
 
This section is meant to prove combinatorial formulas relating multinomial and binomial coefficients, as shown by Proposition~\ref{prop::little_computation} below. These formulas are proved using Newton's General Binomial Theorem. As a consequence we obtain Lemma~\ref{lem::combinatorial_formula} below, which is one of the key ingredients involved in the main results of this article.

\begin{definition}
\label{def::f_function}
Let $f, f': \mathbb{N}^{*} \times \mathbb{N}^{*} \rightarrow \mathbb{N}^{*}$ be defined by $$f(r,k):= k^{r}(k-1)^{r-1} \text{ and } f'(r,k):=(k-1)^{r-1},$$for every $(r,k) \in \mathbb{N}^{*} \times \mathbb{N}^{*}$, where $\mathbb{N}^{*}:= \mathbb{N}\setminus \{0\}$ and with the convention that $0^{0}=1$.

For $n \in \mathbb{N}$, set $\mathit{Sum}(n):= \{ (k_1, \cdots, k_{n-1}) \; \vert \; k_i  \in \mathbb{N} \text{ and } n= \sum\limits_{i=1}^{n-1} k_i \cdot i\}$.  

For $l \geq 2$ and $(k_1, \cdots, k_{l}) \in \mathbb{N}^{l}$ the corresponding \textbf{multinomial coefficient} is defined to be $$\left( \begin{array}{c} k_1+ \cdots +  k_{l}\\ k_1, \cdots, k_{l} \end{array} \right):= \frac{(k_1+ \cdots +  k_{l})!}{k_1! k_{2}!\cdots k_{l}!},$$with the convention that $0!=1$. The multinomial coefficient counts the total number of different words of length  $k_1+ \cdots +  k_{l}$ formed with $l$ distinct letters, where $k_i$ is the multiplicity of the $i^{\text{th}}$ letter. For $l=2$ we obtain the binomial coefficient.

\medskip

For every $r \in \mathbb{N}$ with $r \geq 2$ let $$P_{1,r}(X):= X^{r-2}+\cdots+ X+1$$where $X$ is the variable of the polynomial. Notice that $P_{1,r}(1)= r-1$. Moreover $$P_{1,r}(X)-(r-1)= \sum\limits_{j=1}^{r-2} (X^{j}-1)=(X-1)  \left( \sum\limits_{j=1}^{r-2} (X^{j-1}+\cdots +X+1) \right).$$

For every $r \in \mathbb{N}$ with $r \geq 3$ define $$P_{2,r}(X):=\frac{P_{1,r}(X)-(r-1)}{X-1}= \sum\limits_{j=1}^{r-2} (X^{j-1}+\cdots +X+1). $$ 

More generally, for every $i \in \mathbb{N}$ with $i \geq 2$ and every $ r \in \mathbb{N}$ with $r \geq i+1$ define $$P_{i,r}(X):=\frac{P_{i-1,r}(X)-P_{i-1,r}(1)}{X-1}.$$
\end{definition}

\medskip
\begin{lemma}
\label{lem::P_i_r}
For every $r \in \mathbb{N}$ with $ r \geq 3$ and every $ i \in \{2, \cdots, r-1\} $ we have that
\begin{equation}
\label{equ_P_i_r}
P_{i,r}(X)= \sum\limits_{j=i}^{r-1} P_{i-1,j}(X).
\end{equation} 
\end{lemma}

\begin{proof}
Given $r \geq 3$, we prove the lemma by induction on $i$. First we verify it for $i=2$. Indeed, we have

\begin{equation*}
\begin{split}
P_{2,r}(X)&= \frac{P_{1,r}(X)-P_{1,r}(1)}{X-1}= \frac{X^{r-2}+ X^{r-3}+ \cdots + X+1- (r-1)}{X-1}= \sum\limits_{j=1}^{r-2} \frac{X^j -1}{X-1}\\
&=\sum\limits_{j=2}^{r-1} (X^{j-2}+ \cdots +1)= \sum\limits_{j=2}^{r-1}P_{1,j}(X).
\end{split}
\end{equation*}

Suppose that the formula (\ref{equ_P_i_r}) is true for $i$ and we want to prove it for $i+1$. We have
\begin{equation*}
\begin{split}
P_{i+1,r}(X)&= \frac{P_{i,r}(X)-P_{i,r}(1)}{X-1}= \sum\limits_{j=i}^{r-1} \frac{P_{i-1,j}(X)- P_{i-1,j}(1)}{X-1}= \sum\limits_{j=i+1}^{r-1} P_{i,j}(X),
\end{split}
\end{equation*}as for $j=i$ we have that $P_{i-1,i}(X)=P_{i-1,i}(1)$ because the degree of the polynomial $P_{i-1,i}(X)$ is zero.
\end{proof}

We also need the following easy lemma.

\begin{lemma}
\label{lem::P_i_r_binomial}
For every $j \in \mathbb{N}$ with $j \geq 2$ and every $ i \in \{1, \cdots, j-1\} $ we have that
\begin{equation}
\label{equ_P_i_r_binomial}
P_{i,j}(1)= \left( \begin{array}{c}j-1\\ i \end{array} \right)= \frac{(j-1)!}{i! (j-i-1)!}.
\end{equation} 

\begin{proof}
First notice that $P_{1,j}(1)=j-1=\frac{(j-1)!}{1! (j-2)!}=\left( \begin{array}{c}j-1\\ 1\end{array} \right)$. 

Given $j  \geq 2$, one can easily verify formula (\ref{equ_P_i_r_binomial}) by induction on $i$ and by using formula (\ref{equ_P_i_r}) for $X=1$ and the known equality 

$$\sum\limits_{l=1}^{n} l (l+1)\cdots (l+k)= \frac{n(n+1)\cdots (n+k+1)}{k+2}.$$
\end{proof}

\end{lemma}

\begin{proposition}
\label{prop::little_computation}
For every $r \in \mathbb{N}$ with $r \geq 2$ and every $ i \in \{1, \cdots, r-1\} $ we have that 
\begin{equation}
\label{equ_P_i_r_multinomial}
P_{i,r}(1)= \left( \begin{array}{c}r-1\\ i \end{array} \right)= \sum\limits_{\substack{(k_1, \cdots, k_{r-1}) \in \mathit{Sum}(r) \\  k_1+ \cdots +  k_{r-1}=r-i}} \left( \begin{array}{c} r-i \\ k_1, \cdots, k_{r-1} \end{array} \right).
\end{equation}
\end{proposition}

\begin{proof}
The key ingredient in order to prove the proposition is the following equality, known as Newton's
General Binomial Theorem. Let $n \in \mathbb{N}^{*}$. Formally we have

\begin{equation}
\begin{split}
\label{equ::important_binomial}
(1+X+X^2+\cdots )^n&=\left( \sum\limits_{k=0}^{\infty}X^{k}\right)^n= \left( \frac{1}{1-X}\right)^{n}= (1-X)^{-n}\\
&= \sum\limits_{k=0}^{\infty} (-1)^k \left( \begin{array}{c} -n \\ k \end{array} \right) X^{k}= \sum\limits_{k=0}^{\infty} \left( \begin{array}{c} n+k-1 \\ k \end{array} \right) X^{k},
\end{split}
\end{equation}where by definition $\left( \begin{array}{c} -n \\ k \end{array} \right) :=\frac{(-n)(-n-1)\cdots (-n-k+1)}{k!}$, thus $(-1)^k \left( \begin{array}{c} -n \\ k \end{array} \right)= \frac{n(n+1)\cdots (n+k-1)}{k!}.$

\medskip
Take $n:= r-i$. Then for $k=i$ the coefficient of $X^{i}$ provided by formula (\ref{equ::important_binomial}) is  $\left( \begin{array}{c} r-i+i-1 \\ i \end{array} \right)= \left( \begin{array}{c} r-1 \\ i \end{array} \right)$.

\medskip
Now compute the coefficient of $X^{i}$ with respect to the product $(1+X+X^2+\cdots )^{r-i}$. By the definition of the multinomial coefficient this equals  $ \sum\limits_{\substack{k_2+2\cdot k_3+\cdots +i \cdot k_{i+1}=i \\  k_1+ k_2\cdots +  k_{i+1}=r-i}} \left( \begin{array}{c} r-i \\ k_1, \cdots, k_{i+1} \end{array} \right)$. Notice the following equality of sets:
 
 \begin{equation*}
 \begin{split}
 \{(k_1, k_2,\cdots, k_{i+1}) \; \vert \; k_2+2\cdot k_3+\cdots+ i \cdot k_{i+1}=i \text{ and }  k_1+ k_2\cdots +  k_{i+1}=r-i\}\\
=\\
\{(k_1, k_2,\cdots, k_{i+1}) \; \vert \; (k_1,k_2 \cdots, k_{i+1}, 0, \cdots, 0) \in \mathit{Sum}(r) \text{ and }  k_1+ k_2\cdots +  k_{i+1}=r-i\} \\
=\\
\{(k_1, k_2,\cdots, k_{r-1}) \; \vert \; (k_1,k_2 \cdots, k_{r-1}) \in \mathit{Sum}(r) \text{ and }  k_1+ k_2\cdots +  k_{r-1}=r-i\},\\
\end{split}
\end{equation*}as in the last set, $k_{i+2}=\cdots= k_{r-1}$ are always zero. Therefore
 $$\sum\limits_{\substack{k_2+2k_3+\cdots i k_{i+1}=i \\  k_1+ k_2\cdots +  k_{i+1}=r-i}} \left( \begin{array}{c} r-i \\ k_1, \cdots, k_{i+1} \end{array} \right)= \sum\limits_{\substack{(k_1, \cdots, k_{r-1}) \in \mathit{Sum}(r) \\  k_1+ \cdots +  k_{r-1}=r-i}} \left( \begin{array}{c} r-i \\ k_1, \cdots, k_{r-1} \end{array} \right).$$

\medskip 
We conclude that indeed $\sum\limits_{\substack{(k_1, \cdots, k_{r-1}) \in \mathit{Sum}(r) \\  k_1+ \cdots +  k_{r-1}=r-i}} \left( \begin{array}{c} r-i \\ k_1, \cdots, k_{r-1} \end{array} \right)=  \left( \begin{array}{c} r-1 \\ i \end{array} \right)$.
\end{proof}

\begin{lemma}
\label{lem::combinatorial_formula}

Let $ r, k \in \mathbb{N}$ with $ r,k\geq 1$. We have that
\begin{equation*}
k^{2r-1}- \sum\limits_{(k_1, \cdots, k_{r-1}) \in \mathit{Sum}(r)}\left( \begin{array}{c} k_1+ \cdots +  k_{r-1}\\ k_1, \cdots, k_{r-1} \end{array} \right) f(1,k)^{k_1} f(2,k)^{k_2}\cdots f(r-1,k)^{k_{r-1}}=f(r,k)
\end{equation*}
\end{lemma}

\begin{proof}
By Proposition~\ref{prop::little_computation}, we obtain that
\begin{equation*}
\begin{split}
M&:=k^{2r-1}- \sum\limits_{(k_1, \cdots, k_{r-1}) \in \mathit{Sum}(r)}\left( \begin{array}{c} k_1+ \cdots +  k_{r-1}\\ k_1, \cdots, k_{r-1} \end{array} \right) f(1,k)^{k_1} f(2,k)^{k_2}\cdots f(r-1,k)^{k_{r-1}}\\
&=k^{2r-1}- \sum\limits_{(k_1, \cdots, k_{r-1}) \in \mathit{Sum}(r)}\left( \begin{array}{c} k_1+ \cdots +  k_{r-1}\\ k_1, \cdots, k_{r-1} \end{array} \right) k^{r}(k-1)^{k_2+2k_3+\cdots+ (r-2) k_{r-1}}\\
&=k^{r}\left(  k^{r-1}-  \sum\limits_{(k_1, \cdots, k_{r-1}) \in \mathit{Sum}(r)}\left( \begin{array}{c} k_1+ \cdots +  k_{r-1}\\ k_1, \cdots, k_{r-1} \end{array} \right) (k-1)^{k_2+2k_3+\cdots+ (r-2) k_{r-1}} \right) \\
&= k^{r}\left(  k^{r-1} -1- \sum\limits_{i=2}^{r-1} \ \ \sum\limits_{\substack{(k_1, \cdots, k_{r-1}) \in \mathit{Sum}(r) \\  k_1+ \cdots +  k_{r-1}=i}} \left( \begin{array}{c} i \\ k_1, \cdots, k_{r-1} \end{array} \right) (k-1)^{r-i} \right)\\
&=k^{r}\left(  k^{r-1} -1-  \sum\limits_{i=2}^{r-1} P_{r-i, r}(1) (k-1)^{r-i} \right) \\
\end{split}
\end{equation*}
\begin{equation*}
\begin{split}
&=k^{r} \left( (k-1)(k^{r-2}+\cdots+ k+1)-  \sum\limits_{i=2}^{r-1} P_{r-i, r}(1)  (k-1)^{r-i} \right)\\
&=k^{r}(k-1)\left(P_{1,r}(k)- P_{1,r}(1)-\sum\limits_{i=2}^{r-2} P_{r-i, r}(1) (k-1)^{r-i-1} \right)\\
&=k^{r}(k-1)^2\left(P_{2,r}(k)-P_{2,r}(1) -\sum\limits_{i=2}^{r-3} P_{r-i, r}(1) (k-1)^{r-i-2} \right)\\
&=\cdots\\
&=k^{r}(k-1)^{r-3} \left(P_{r-3,r}(k)-P_{r-3,r}(1) -\sum\limits_{i=2}^{r-1-(r-3)} P_{r-i, r}(1) (k-1)^{r-i-(r-3)} \right)\\
&=k^{r}(k-1)^{r-3} \left((k-1)P_{r-2,r}(k) -\sum\limits_{i=2}^{2} P_{r-i, r}(1) (k-1)^{r-i-(r-3)} \right)\\
&=k^{r}(k-1)^{r-3} \left((k-1)P_{r-2,r}(k) -P_{r-2, r}(1) (k-1)\right)\\
&=k^{r}(k-1)^{r-2} (k-1)P_{r-1,r}(k)=k^{r}(k-1)^{r-1}P_{r-1,r}(1)= k^{r}(k-1)^{r-1} \left( \begin{array}{c} r-1 \\ r-1 \end{array} \right)\\
&=k^{r}(k-1)^{r-1}= f(r,k).
\end{split}
\end{equation*}
\end{proof}

\section{Computation of the Hecke algebra}
\label{sec::Hecke_alg}

Every function of the Hecke algebra $C_{c}(\G,K)$ is a finite linear combination of functions of the form $\mathbbm{1}_{K g K}$ with $g \in \G$. To better understand the Hecke algebra $C_{c}(\G,K)$ one should be able to evaluate the convolution product of a finite number of functions $\mathbbm{1}_{K g K}$ with $g \in \G$. The structure of the group $\G$ allows us to characterize its $K$--double cosets and to perform computations of the desired convolution products.

\subsection{$K$--double cosets of $\G$}
\label{subsec::indep_hyp_elements}

The goal of this subsection is to describe and count $K$--double cosets of $\G$. Along the way we also give a description of some special hyperbolic elements of $\G$ that play an important role in that study. 

\medskip
When $F$ is primitive but not $2$--transitive, the group $\G$ still has some of the properties of closed non-compact subgroups of $\Aut(\T)$ that act $2$--transitively on the boundary $\bd \T$.

\begin{remark}
\label{rem::existence_hyp_elements}
As $F$ is primitive, given an edge $e'$ of $\T $ at odd distance from $e$, one can construct (using the definition of $\G$) a hyperbolic element in $\G$ translating $e$ to $e'$. Moreover, every hyperbolic element in $\G$ has even translation length. 
\end{remark}

\begin{lemma}($KA^{+}K$ decomposition)
\label{lem::KAK_decomposition}
Let $F$ be primitive. Then $\G$ admits a $KA^{+}K$ decomposition, where $$A^{+}:=\{\gamma \in \G \; \vert \; \gamma \text{ is hyperbolic, } e \subset \Min(\gamma), \gamma \text{ translates the edge } e \text{ inside } \T_{x,e} \} \cup \{\id\}.$$
\end{lemma}

\begin{proof}
Let $g \in \G$. If $g(x)=x$, then $g \in K$. Suppose that $g(x) \neq x$. Consider the geodesic segment $[x, g(x)]$ in $\T$ and denote by $e_1$ the edge of the star of $x$ that belongs to $[x, g(x)]$. Notice that $[x, g(x)]$ has even length and that there exists $k \in K$ such that $k(e_1)=e$; therefore, $kg(x) \in \T_{x,e}$. By Remark~\ref{rem::existence_hyp_elements} there is a hyperbolic element  $\gamma \in \G$ of translation length equal to the length of $[x, g(x)]$ that translates the edge $e$ inside $\T_{x,e}$ and such that $\gamma(x)=kg(x)$; thus $\gamma^{-1}kg \in K$. Notice that the $KA^{+}K$ decomposition of an element $g \in \G$ is not unique.
\end{proof}

\begin{remark}
\label{rem::K_cosets}
By Lemma~\ref{lem::KAK_decomposition} one notices that every $K$--double coset of $\G$ is of the form $K\gamma K$ for some $\gamma \in A^{+}$. When $F$ is $2$--transitive we have that $A^{+}=<a>^{+}$, where $a$ is a hyperbolic element of $\G$, with $\vert a \vert=2$ and $x \in \Min(a)$. In this case we obtain the well-known polar decomposition of $\G$ and the $K$--double cosets of $\G$ are just $\{Ka^{n}K\}_{n \geq 0}$.

\end{remark}

In order to describe the hyperbolic elements in $A^{+}$ we need the following easy but important remarks.

\begin{remark}
\label{rem::prod_of_two_hyp}
Let $\gamma, \gamma' \in A^{+}$ be two hyperbolic elements. We claim that $\gamma \gamma'$ is again a hyperbolic element of $A^{+}$ and that $\vert \gamma \vert +\vert \gamma' \vert = \vert \gamma \gamma' \vert$. Indeed, the edge $e$ endowed with the orientation pointing towards the boundary $\bd \T_{x, e}$ is sent by any $\eta \in A^{+}$ to an edge of $\T_{x, e}$ with induced orientation that points towards the boundary $\bd \T_{x, e}$. This observation proves the claim. 
\end{remark}

\begin{remark}
\label{rem::number_orbits}
For $i \in  \{1,\cdots, d\}$ let $F_i$ be the  stabilizer in $F$ of $i$. As $F$ is transitive on the set $\{1,\cdots, d\}$ the number of orbits of $F_i$ acting on the set $\{1,\cdots, d\}\setminus \{i\}$ is independent of the choice of the color $i$. Therefore, we denote by $\kk$ the total number of all such orbits. In addition, independently of the choice of a color $i \in \{1,\cdots, d\}$ we are allowed to denote by $n_j$, where $j \in \{1,\cdots, \kk \}$, the number of elements of each $F_{i}$--orbit in $\{1,\cdots, d\}\setminus \{i\}$. Notice that $n_1+\cdots +n_{\kk}=d-1$.  If $F$ is $2$--transitive $\kk=1$. If $F$ is not $2$--transitive we have that $\kk \geq 2$. 
\end{remark} 

Let us now describe the elements in $A^{+}$ of translation length $2$. By Remarks~\ref{rem::existence_hyp_elements} and~\ref{rem::number_orbits} it is easy to see that there are exactly $\kk$ distinct $K$--double cosets $K\gamma K$, with $\gamma \in A^{+}$ and $\vert \gamma \vert =2$.  Let $K\gamma K$ be such a $K$--double coset and let $\gamma_1, \gamma_2 \in A^{+} \cap K \gamma K$. Then $\vert \gamma_i\vert =2$, for every $i \in \{1,2\}$. We notice that there are two cases: either there is $k \in \G_{e}$ such that $k\gamma_{1}(e)=\gamma_2(e)$, or $k\gamma_{1}(e) \neq  \gamma_2(e)$ for every $k \in K$. Moreover, the latter mentioned case occurs if and only if $F$ is not $2$--transitive.

\begin{definition}
\label{def::primitive_pairs}
Let $F$ be primitive and let $\gamma_1, \gamma_2 \in A^{+} $ be such that $\vert \gamma_1\vert =\vert \gamma_2 \vert=2$. We say that \textbf{$(\gamma_1,\gamma_2)$} is a \textbf{$K$--primitive pair} if there exists $k \in K$ such that $k \gamma_{1}(e)=\gamma_2(e)$. If this is the case, notice that $k \in \G_{e}$ and $K\gamma_1K=K\gamma_2K$. In particular, for every $k' \in \G_{[x,\gamma_{2}(e)]} k \G_{[x,\gamma_{1}(e)]}$ we have that $k' \gamma_{1}(e)=\gamma_2(e)$.

\end{definition}

\begin{lemma}
\label{lem::number_primitive_pairs}
Let $F$ be primitive. Then every $K$--double coset $K \gamma K$, with $\gamma \in A^+$ of translation length $2$, admits exactly $\kk$ $K$--primitive pairs up to $K$--left-right multiplication. In particular there are exactly $\kk^{2}$ $K$--primitive pairs of $A^{+}$, up to $K$--left-right multiplication.
\end{lemma}

\begin{proof}
The proof follows from Remark~\ref{rem::number_orbits}.
\end{proof}

For a further use and following Lemma~\ref{lem::number_primitive_pairs} let us make the following notation.
\begin{notation}
\label{not::classes_primitive}
For every $\gamma \in A^+$, with $\vert \gamma \vert=2$, we denote by $\{\gamma^{j}\}_{j \in \{1, \cdots, \kk \} }\subset A^+$ the set of different elements of translation length $2$ such that for every $\eta \in K\gamma K  \cap A^+$ there exists a unique $j \in \{1, \cdots, \kk \}$ with the property that $(\eta, \gamma^{j})$ is a $K$--primitive pair.
\end{notation}

\begin{lemma}
\label{lem::product_two_primitive}
Let $F$ be primitive. Let $\gamma,\gamma' \in A^{+}$ be such that there exists $k \in K$ with $k\gamma(e)=\gamma'(e)$ and let $(\gamma_1, \gamma_1')$ be a $K$--primitive pair. Respectively, let $\gamma_2, \gamma_2' \in A^{+}$ with $\vert \gamma_2 \vert =\vert \gamma_2' \vert$ and such that $K\gamma_2K=K\gamma_2'K$. Then there exists $k' \in K$ with the property that $k'\gamma\gamma_1(e)=\gamma' \gamma_1'(e)$, respectively, $k'\gamma\gamma_2(x)=\gamma' \gamma_2'(x)$.
\end{lemma}

\begin{proof}
We have to find $k' \in K$ such that $k'$ sends the geodesic segment $[x, \gamma \gamma_1(e)]$ into the geodesic segment $[x, \gamma' \gamma_1'(e)]$. Notice that by hypothesis, it already exists $k \in K$ that sends the geodesic segment $[x, \gamma(e)] \subset [x, \gamma \gamma_1(e)]$ into the geodesic segment $[x, \gamma'(e)] \subset [x, \gamma' \gamma_1'(e)]$. We just need to find $k' \in \G_{[x, \gamma'(e)]}$ that sends $[\gamma'(e), k\gamma\gamma_1(e)]$ into $[\gamma'(e), \gamma' \gamma_1'(e)]$. Suppose the contrary that there is no such $k' \in \G_{[x, \gamma'(e)]}$. Using the definition of $\G$ (or the fact that $\G$ has Tits' independence property, see~\cite{Amann}), we have that by fixing the edge $\gamma'(e)$, the segment $[\gamma'(e), \gamma' \gamma_1'(e)]$ is not contained in the $\G_{\gamma'(e)}$--orbit of the segment $[\gamma'(e), k\gamma\gamma_1(e)]$. Apply now $(\gamma')^{-1}$. The latter assumption becomes that by fixing the edge $e$, the segment $[e, \gamma_1'(e)]$ is not contained in the $\G_{e}$--orbit of the segment $[e, (\gamma')^{-1} k\gamma \gamma_1(e)]$. But $(\gamma')^{-1} k\gamma \in \G_{e} < K$ and recall that $(\gamma_1, \gamma_1')$ is a $K$--primitive pair. We obtained thus a contradiction and the conclusion follows.

If we just consider that $K\gamma_2K=K\gamma_2' K$, one proceeds in the same way as above by replacing the edge $e$ with the vertex $x$ where is needed.

\end{proof}

The importance of $K$--primitive pairs is expressed by the next proposition where two $K$--double cosets of $\G$ are compared. Its proof relies on the fact that every element in $A^{+}$ can be written as a product of elements of $A^{+}$ that are of translation length $2$; the aim of describing the elements of $A^{+}$ being thus achieved.

\begin{proposition}
\label{prop::product_of_trans_lenght_two}
Let $F$ be primitive and let $\gamma \in A^{+}$. Then $\gamma= \gamma_1\gamma_2\cdots \gamma_t$, where $\gamma_i \in A^{+}$ is such that $\vert \gamma_i \vert =2$, for every $i \in \{1,\cdots, t\}$.  Moreover, for any $\gamma_1,\cdots, \gamma_t, \gamma'_1, \cdots, \gamma'_t \in A^{+}$, with $\vert \gamma_i \vert =\vert \gamma'_i \vert=2$ for every $i \in \{1,\cdots, t\}$, the following equivalence is true: $K\gamma_1\gamma_2\cdots \gamma_t K= K\gamma'_1\gamma'_2\cdots \gamma'_t K$ if and only if $(\gamma_i, \gamma'_i)$ are $K$--primitive for every $i \in \{1, \cdots, t-1\}$ and $K\gamma_tK= K \gamma'_t K$.
\end{proposition}

\begin{proof}
Let us prove the first part of the proposition. Notice that if $\vert \gamma \vert =2$ the conclusion follows. Assume that $\vert \gamma \vert > 2$  and denote the vertices of the geodesic $[x, \gamma(x)] \subset \T_{x,e}$ by $x=x_0, x_1,x_2,\cdots, x_{2t} $. Then by Remark~\ref{rem::existence_hyp_elements} there exists $\gamma_1 \in A^{+}$ with $\vert \gamma_1 \vert =2$ such that $\gamma_1(e)= [x_2,x_3]$. Notice that $\gamma_1^{-1} \gamma $ is still an element of $A^{+}$, with $\vert \gamma_1^{-1} \gamma  \vert =\vert \gamma \vert-2$. Indeed, this is because the orientation of the edge $\gamma_1^{-1} \gamma(e)$ induced from the orientation of $e$ that points towards the boundary of the half-tree $\T_{x,e}$ is preserved. Apply now the above procedure to $\gamma_1^{-1} \gamma$. There exists thus $\gamma_2 \in A^{+}$ with $\vert \gamma_2\vert =2$ and such that $\gamma_2^{-1}\gamma_1^{-1} \gamma$ is in $A^{+}$ with $\vert \gamma_2^{-1}\gamma_1^{-1} \gamma \vert = \vert \gamma \vert-4$. By induction we obtain the elements $\gamma_1, \gamma_2, \cdots, \gamma_t$ with the property that $\gamma= \gamma_1\gamma_2\cdots \gamma_t$. The conclusion follows.

Let us prove the second part of the proposition.

Assume that $K\gamma_1\gamma_2\cdots \gamma_t K= K\gamma'_1\gamma'_2\cdots \gamma'_t K$. This means that there exist $k,k' \in K$ such that $k \gamma_1\gamma_2\cdots \gamma_t k'=\gamma'_1\gamma'_2\cdots \gamma'_t$; this is equivalent to saying that there is $k \in K$ such that the geodesic segment $[x, \gamma_1\gamma_2\cdots \gamma_t(x)]$ is sent by $k$ into the geodesic segment $[x, \gamma'_1\gamma'_2\cdots \gamma'_t(x)]$. From here we deduce that in fact $k\gamma_1(e)= \gamma'_1(e) \subset [x, \gamma'_1\gamma'_2\cdots \gamma'_t(x)]$. Therefore, $(\gamma'_1)^{-1}k\gamma_1=:k_1 \in \G_{e}< K$ and $k_1\gamma_2\cdots \gamma_t k'=\gamma'_2\cdots \gamma'_t$. We apply again the above procedure and by induction we obtain that $(\gamma_i, \gamma'_i)$ are $K$--primitive for every $i \in \{1, \cdots, t-1\}$ and $K\gamma_tK= K \gamma'_t K$. The implication follows.

Assume now that $(\gamma_i, \gamma'_i)$ are $K$--primitive for every $i \in \{1, \cdots, t-1\}$ and $K\gamma_tK= K \gamma'_t K$. We need to find $k \in K$ such that $k([x, \gamma_1\gamma_2\cdots \gamma_t(x)])=[x, \gamma'_1\gamma'_2\cdots \gamma'_t(x)]$. We prove this by induction on $i \in \{1, \cdots, t\}$ by applying Lemma~\ref{lem::product_two_primitive}. Notice that for $i=1$ we just use that $(\gamma_1, \gamma'_1)$ is a $K$--primitive pair. Suppose that the induction step is true for $i$ and we want to prove it for $i+1$. This means that there is $k_i \in K$ that sends the geodesic segment $[x, \gamma_1\gamma_2\cdots \gamma_i(e)]$ into the geodesic segment $[x, \gamma'_1\gamma'_2\cdots \gamma'_i(e)]$ and we want to find $k_{i+1} \in K$ that sends the geodesic segment $[x, \gamma_1\gamma_2\cdots \gamma_i \gamma_{i+1}(e)]$ into the geodesic segment $[x, \gamma'_1\gamma'_2\cdots \gamma'_i\gamma_{i+1}(e)]$. This follows by applying Lemma~\ref{lem::product_two_primitive}, even for the last induction step when $i+1=t$. The conclusion is thus proved.

\end{proof}

The following lemma counts the $K$--double cosets of $\G$. Together with Proposition~\ref{prop::product_of_trans_lenght_two} the goal of this subsection is achieved.

\begin{lemma}
\label{lem::existence_indep_hyp_elements}
Let $F$ be primitive. Using Convention~\ref{not::general_notation}, for every $r \geq 1$ the total number of orbits of the group $\G_e$ acting on $V_{x,2r}$ is exactly $\kk^{2r-1}$.

For any two distinct $\G_e$--orbits $[x_1], [x_2]$ of the $2r$--level $V_{x,2r}$, there exist hyperbolic elements $\gamma_1, \gamma_2 \in A^{+}$ such that $\gamma_i(x)=x_i$, for $i \in \{1,2\}$. In addition, for any such choice of $\gamma_1, \gamma_2$ having the above properties we have that $K \gamma_1 K \neq K \gamma_2 K$. 

Therefore, for every $r \geq 1$ there are exactly $\kk^{2r-1}$ disjoint $K$--double cosets $K\gamma K$, where $\gamma \in A^{+}$ and $\vert \gamma \vert =2r$. 
\end{lemma}

\begin{proof}
We prove the first assertion by induction on $r$. For $r=1$, this is Remark~\ref{rem::number_orbits}. 

Suppose now that the induction hypothesis is true for $r \geq1$ and we want to prove it for $r+1$. Fix any vertex $y$ of the $2r$--level $V_{x,2r}$ and denote by $[z,y]$ the corresponding last edge of the geodesic segment $[x,y]$. We claim that there are exactly $\kk^{2}$ orbits of $\G_{[z, y]}$ acting on $V_{z, 3}$.  Indeed, this follows from the fact that $(d-1)^2=(n_1+\cdots +n_{\kk})^2=\sum_{i,j=1}^{\kk}n_i n_j$, where we have $\kk^{2}$ terms as the pairs $(i,j)$ and $(j,i)$ give disjoint orbits of the group $\G_{[z, y]}$ acting on $V_{z, 3}$. Applying then Tits' independence property of $\G$ for the subgroup $\G_{[z, y]}$, we obtain that in fact there are $\kk^{2}$ orbits of $\G_{[x, y]}$ acting on $V_{z, 3}$.  From the above claim the induction step follows easily.

The first part of the second assertion of the lemma is Remark~\ref{rem::existence_hyp_elements}.  Notice that for a given $\G_e$--orbit $[x_1]$, the hyperbolic element $\gamma_1$ does not depend, up to $\G_e$--conjugation, on the chosen representative $x_1$.

It remains to prove that for any two such hyperbolic elements $\gamma_1, \gamma_2$ we have that $K \gamma_1 K \neq K \gamma_2 K$. Indeed, consider two disjoint $\G_e$--orbits $[x_1], [x_2]$ of the $2r$--level $V_{x,2r}$ and two hyperbolic elements $\gamma_1, \gamma_2 \in \G$ such that $\gamma_i(x)=x_i$, for $i \in \{1,2\}$. By contraposition, suppose that $(K\gamma_1 K) \cap (K\gamma_2K) \neq \emptyset$.  Then $\gamma_1= k_1\gamma_2 k_2$, for some $k_1,k_2 \in K$. As $\gamma_i(x)=x_i$ and $x_1, x_2 \in V_{x,2r}$, we have that $x_1=k_1(x_2)$ and $k_1 \in \G_e$. This is a contradiction with $[x_1] \neq [x_2]$ as $\G_e$--orbits. The conclusion follows.
\end{proof}

\subsection{Convolution products of $K$--double cosets}
\label{subset::product_two_double_classes}

Recall that our main goal is to prove that the Hecke algebra $C_{c}(\G, K)$ is infinitely generated when $F$ is primitive but not $2$--transitive. Therefore, we need to understand the elements $f \in C_{c}(\G,K)$ and convolution products of those. As the support of any function $f \in C_{c}(\G,K)$ is compact and $K$--bi-invariant, using Lemma~\ref{lem::KAK_decomposition} the support of $f$ is covered with a finite number of compact-open subsets of the form $K \eta K$, with $\eta \in A^+$. We obtain that 
\begin{equation}
\label{equ::gen_decomposition}
f= \sum\limits_{i \in I} a_i \mathbbm{1}_{K\eta_iK},
\end{equation} where $I$ is a finite set, $a_i \in \mathbb{C}$ and $\eta_i \in A^{+}$ for every $i \in I$. This implies that the set of all functions $\mathbbm{1}_{K\gamma K}$, with $\gamma \in A^+$, forms an infinite (countable) base for the $\mathbb{C}$--vector space $C_{c}(\G,K)$ endowed with the addition of functions. 

To study convolution products of elements in $C_{c}(\G,K)$ it is enough to compute convolution products of the form $\mathbbm{1}_{K\gamma_1K} \ast \mathbbm{1}_{K\gamma_2K} \ast \cdots \ast \mathbbm{1}_{K\gamma_nK}$, where $\gamma_i \in A^+$ for every $i \in \{1, \cdots, n\}$. Firstly, given two elements $\gamma_1, \gamma_2 \in A^+$ not necessarily of the same translation length, we want to evaluate the $K$--double cosets appearing in the decomposition~(\ref{equ::gen_decomposition}) of the convolution product $\mathbbm{1}_{K\gamma_1K} \ast \mathbbm{1}_{K\gamma_2K}$.

Recall that by definition $$ \mathbbm{1}_{K\gamma_1K} \ast \mathbbm{1}_{K\gamma_2K}(x):= \int_{G}  \mathbbm{1}_{K\gamma_1K}(xg )\mathbbm{1}_{K\gamma_2K}(g^{-1})d\mu(g).$$ In order to determine the support of $ \mathbbm{1}_{K\gamma_1K} \ast \mathbbm{1}_{K\gamma_2K}$, we need to have that $xg \in K\gamma_1K$  and $g^{-1} \in K\gamma_2K$. This gives that $x \in  K \gamma_1K\gamma_2 K$, implying that the support of the function $ \mathbbm{1}_{K\gamma_1K} \ast \mathbbm{1}_{K\gamma_2K}$ is contained in $K \gamma_1K\gamma_2 K$. It is easy to check that in fact the support of $ \mathbbm{1}_{K\gamma_1K} \ast \mathbbm{1}_{K\gamma_2K}$ equals $K \gamma_1K\gamma_2 K$. 


\begin{definition}
\label{def::maximal_K_double_cosets}
Let $X$ be a compact $K$--bi-invariant subset of $\G$. Notice that $X$ is covered with a finite number of (open) $K$--double cosets $K \eta K$, where $\eta \in A^+$. We say that $K \eta K$ is a \textbf{maximal $K$--double coset} of $X$, where $\eta \in A^+$, if the translation length $\vert \eta \vert$ is maximal among all $K$--double cosets that appear in the above decomposition of $X$. 

Let $\gamma_1, \gamma_2 \in A^{+}$. Let $\mathbbm{1}_{K\gamma_1K} \ast \mathbbm{1}_{K\gamma_2K}= \sum\limits_{i \in I} a_i \mathbbm{1}_{K\eta_iK}$ as in equality~(\ref{equ::gen_decomposition}). We say that $a_i$ is a \textbf{maximal coefficient} of $\mathbbm{1}_{K\gamma_1K} \ast \mathbbm{1}_{K\gamma_2K}$ if the corresponding $K$--double coset $K \eta_iK$ of the support of  $\mathbbm{1}_{K\gamma_1K} \ast \mathbbm{1}_{K\gamma_2K}$ is maximal.
\end{definition}

\begin{proposition}
\label{prop::KKK}
Let $F$ be primitive and let $\gamma_1, \gamma_2 \in A^{+}$. By Proposition~\ref{prop::product_of_trans_lenght_two}, let $\gamma_1= \gamma'_1\gamma'_2\cdots \gamma'_{t_1}$ and $\gamma_2= \gamma''_1\gamma''_2\cdots \gamma''_{t_2}$,  where $\gamma'_i , \gamma''_j \in A^{+}$ are such that $\vert \gamma'_i \vert=\vert \gamma''_j \vert =2$, for every $i \in \{1,\cdots, t_1\}$ and every $j \in \{1,\cdots, t_2\}$.

Then the maximal $K$--double cosets that appear in $K\gamma_1 K \gamma_2 K$ are of the form $$K\gamma'_1\gamma'_2\cdots \gamma'_{t_1-1} \ \gamma \ \gamma''_1\gamma''_2\cdots \gamma''_{t_2}K$$ where $\gamma$ can be any element of  $A^{+} \cap K\gamma'_{t_1} K$.  In particular, we obtain exactly $\kk$ such maximal $K$--double cosets. 
\end{proposition}

\begin{proof}
Let $(\xi_{i-}, \xi_{i+})$ be the translation axis of $\gamma_i$, for $i \in \{1,2\}$. As $K \gamma_1 \gamma_2 K$ is evidently a maximal $K$--double coset of $K\gamma_1 K \gamma_2 K$, it remains to find all other $K$--double cosets $K \eta K$ such that $\eta \in A^+$ and $\vert \eta \vert= \vert \gamma_1 \vert + \vert \gamma_2 \vert$.

To compute the maximal $K$--double cosets that appear in the decomposition of $K\gamma_1 K \gamma_2 K$, it is enough to study elements of the form $\gamma_1 k \gamma_2$ for $k \in K$ with the property that $k [x, \xi_{2+}] \cap [\xi_{1-},\xi_{1+}] \subset  [x, \xi_{1+}]$. Indeed, if $k[x, \xi_{2+}] \cap [\xi_{1-},\xi_{1+}]$ was not just equal to $x$ and not a subset of $(x, \xi_{1+}]$, one would see that $K\gamma_1 k \gamma_2K$ is not a maximal $K$--double coset of $K\gamma_1 K \gamma_2 K$.

Consider first the case when $k \in \G_{e}$. For such $k$ we claim that $K\gamma_1 k \gamma_2 K=K\gamma_1 \gamma_2 K$. Indeed, we apply  Lemma~\ref{lem::product_two_primitive}, as we need to find $k_1 \in \G_e$ such that $k_1 \gamma_1 k \gamma_2(x)= \gamma_1 \gamma_2 (x)$. The claim follows. In addition, notice that such $k_1$ fixes point-wise the geodesic segment $[x, \gamma_1(e)]$. 

Consider now the case when $k \in K$ such that $k(e) \notin  (\gamma_{1-}, \gamma_{1+})$. Notice that for every edge in the star of $x$, which is not on the bi-infinite geodesic $(\gamma_{1-}, \gamma_{1+})$, there exists some $k \in K$ sending $e$ to it. For such $k \in K$ we need to decompose the element  $\gamma_1 k \gamma_2$, which is in $A^{+}$, using elements $\gamma \in A^{+}$ with $\vert \gamma \vert =2$. By the hypothesis on $k$ we have that $\vert \gamma_1 k \gamma_2 \vert= \vert \gamma_1 \vert + \vert \gamma_2 \vert= t_1 +t_2$. There exist thus $\gamma', \gamma'' \in A^{+}$ such that $\gamma_1 k \gamma_2 = \gamma' \gamma''$, with $\vert \gamma' \vert= \vert \gamma_1 \vert $ and $\vert \gamma'' \vert =\vert \gamma_2 \vert$. Because of the choice of $k$, we have that $\gamma_1(x)=\gamma'(x)$; thus $(\gamma')^{-1} \gamma_1 \in K$ and $\gamma'(e)$ belongs to the star of $\gamma_1(x)$, being different from $\gamma_{1}(e)$ and the edge belonging to the geodesic segment $[x, \gamma_1(x)]$. Therefore, the first $t_{1}-1$ terms of the decomposition of $\gamma_1 k \gamma_2$ are $\gamma_1' \gamma_2' \cdots \gamma_{t_1-1}'$ and the $t_1$--term appearing in the decomposition of $\gamma_1 k \gamma_2$ is an element $\gamma \in A^{+} \cap K\gamma_{t_1}'K$. Notice that this $t_1$--term depends strictly on the element $k$ and up to $K$-left-right multiplication, every element of $A^{+} \cap K\gamma_{t_1}'K$ appears as a $t_1$--term of $\gamma_1 k \gamma_2$. From the equality $\gamma_1 k \gamma_2 = \gamma' \gamma''$ and the above properties we obtain that $K \gamma_2 K= K \gamma'' K$. Following Lemma~\ref{lem::product_two_primitive}, we conclude that there exists $k' \in K$ such that $k' \gamma' \gamma''(x)= \gamma' \gamma_2(x)$. 

\medskip
Combining the two cases studied above and using Proposition~\ref{prop::product_of_trans_lenght_two}, we obtain that $K \gamma_1 k \gamma_2 K= K \gamma_1' \gamma_2' \cdots \gamma_{t_1-1}' \gamma \gamma''_1\gamma''_2\cdots \gamma''_{t_2}K$, where $\gamma \in A^{+} \cap K \gamma'_tK$. Moreover, up to $K$--primitivity, there are $\kk$ such $\gamma$ elements in $A^{+} \cap K \gamma'_tK$ and the last part of the proposition follows. 
\end{proof}



The next proposition computes the maximal coefficients appearing in $\mathbbm{1}_{K\gamma_1K} \ast \mathbbm{1}_{K\gamma_2K}$.

\begin{proposition}
\label{prop::max_coeff}
Let $F$ be primitive and let $\gamma_1, \gamma_2 \in A^{+}$. By Proposition~\ref{prop::product_of_trans_lenght_two}, let $\gamma_1= \gamma'_1\gamma'_2\cdots \gamma'_{t_1}$ and $\gamma_2= \gamma''_1\gamma''_2\cdots \gamma''_{t_2}$,  where $\gamma'_i , \gamma''_j \in A^{+}$ are such that $\vert \gamma'_i \vert=\vert \gamma''_j \vert =2$, for every $i \in \{1,\cdots, t_1\}$ and every $j \in \{1,\cdots, t_2\}$.

Then every maximal coefficient appearing in $\mathbbm{1}_{K\gamma_1K} \ast \mathbbm{1}_{K\gamma_2K}$ equals $\mu (K)$, where $\mu$ is the left Haar measure on $\G$. In particular, by normalizing $\mu(K)=1$, we have that all maximal coefficients equal one.
\end{proposition}

\begin{proof}
As in equality~(\ref{equ::gen_decomposition}) we have that $\mathbbm{1}_{K\gamma_1K} \ast \mathbbm{1}_{K\gamma_2K}= \sum\limits_{i \in I} a_i \mathbbm{1}_{K\eta_iK}$, where $I$ is a finite set, $a_i \in \mathbb{C}$ and $\eta_i \in A^{+}$ for every $i \in I$. Let $K\eta_iK$ be a maximal $K$--double coset and let $h \in K\eta_iK$. By Proposition~\ref{prop::KKK} $$\eta_i= \gamma'_1\gamma'_2\cdots \gamma'_{t_1-1}  \ \gamma \ \gamma''_1\gamma''_2\cdots \gamma''_{t_2}=\gamma_1 (\gamma'_{t_1})^{-1} \gamma \gamma_2$$ where $\gamma$ is an element of  $A^{+} \cap K\gamma'_{t_1} K$. In particular, using Proposition~\ref{prop::product_of_trans_lenght_two}, we can suppose, without loss of generality, that $\gamma(x)= \gamma_{t_1}'(x)$. 

We want to compute $\mathbbm{1}_{K\gamma_1K} \ast \mathbbm{1}_{K\gamma_2K}(h)$ which equals $a_i$. As all functions involved are $K$--bi-invariant, we can suppose that $h= \eta_i$. It remains to evaluate $$\int_{G}  \mathbbm{1}_{K\gamma_1K}(\eta_i g )\mathbbm{1}_{K\gamma_2K}(g^{-1})d\mu(g).$$ This reduces to find all $g \in K \gamma_2^{-1}K$ such that $\eta_i g \in K \gamma_1K$; this is equivalent to evaluate the intersection $\eta_i^{-1} K \gamma_1 K \cap K \gamma_2^{-1}K$. We would have that 
\begin{equation*}
\begin{split}
\mathbbm{1}_{K\gamma_1K} \ast \mathbbm{1}_{K\gamma_2K}(\eta_i)&=a_i= \mu(\eta_i^{-1} K \gamma_1 K \cap K \gamma_2^{-1}K )\\
&= \mu(\gamma_2^{-1} (\gamma_1 (\gamma'_{t_1})^{-1} \gamma)^{-1} K \gamma_1 K \cap K \gamma_2^{-1}K )\\
&= \mu( \gamma^{-1} (\gamma'_1\gamma'_2\cdots \gamma'_{t_1-1})^{-1} K \gamma_1 K \cap \gamma_2 K \gamma_2^{-1}K ).
\end{split}
\end{equation*}

Notice the following. For every $g \in \gamma_2 K \gamma_2^{-1}K$, we have $$\dist_{\T}(g(x), \gamma_2(x))= \dist_{\T}( \gamma_2 k_1 \gamma_2^{-1}k_2 (x), \gamma_2(x))=\vert \gamma_2\vert .$$ For $g \in  \gamma^{-1} (\gamma'_1\gamma'_2\cdots \gamma'_{t_1-1})^{-1} K \gamma_1 K$ we have $$\dist_{\T}(g(x), \gamma^{-1} (\gamma'_1\gamma'_2\cdots \gamma'_{t_1-1})^{-1}(x))= \dist_{\T}(k_3 \gamma_1(x) , x)= \vert \gamma_1 \vert.$$ 

As $\gamma_2(x) \in \T_{x, e}$ and $\gamma^{-1} (\gamma'_1\gamma'_2\cdots \gamma'_{t_1-1})^{-1}(x) \in \T \setminus \T_{x, e}$, and because $\vert \gamma^{-1} (\gamma'_1\gamma'_2\cdots \gamma'_{t_1-1})^{-1} \vert= \vert \gamma_1^{-1} \vert $ we obtain that $$\overline{B}(\gamma^{-1} (\gamma'_1\gamma'_2\cdots \gamma'_{t_1-1})^{-1}(x), \vert \gamma_1\vert ) \cap \overline{B}(\gamma_2(x), \vert \gamma_2\vert )=\{x\}$$ where $\overline{B}(y, r) \subset \T$ denotes the closed ball centered at the vertex $y$ and of radius $r$. This implies that for every $g \in \gamma^{-1} (\gamma'_1\gamma'_2\cdots \gamma'_{t_1-1})^{-1} K \gamma_1 K \cap \gamma_2 K \gamma_2^{-1}K$ we necessarily have that $g(x)=x$; therefore, $\gamma^{-1} (\gamma'_1\gamma'_2\cdots \gamma'_{t_1-1})^{-1} K \gamma_1 K \cap \gamma_2 K \gamma_2^{-1}K \subset K$. In fact, we claim that $$\gamma^{-1} (\gamma'_1\gamma'_2\cdots \gamma'_{t_1-1})^{-1} K \gamma_1 K \cap \gamma_2 K \gamma_2^{-1}K = K.$$ Indeed, it is immediate that $K \subset \gamma_2 K \gamma_2^{-1}K $. It remains to show that $K \subset \gamma^{-1} \gamma_{t_1}' \gamma_1^{-1} K \gamma_1 K$. As we have supposed that $\gamma(x)= \gamma_{t_1}'(x)$, we have that $ \gamma^{-1}\gamma_{t_1}'  \in K$, so $\gamma^{-1} \gamma_{t_1}' \gamma_1^{-1} \gamma_1 K= K$. This proves the claim and we obtain that $$a_i=\mu( \gamma^{-1} (\gamma'_1\gamma'_2\cdots \gamma'_{t_1-1})^{-1} K \gamma_1 K \cap \gamma_2 K \gamma_2^{-1}K )=\mu(K).$$ The conclusion follows.

\end{proof}

\section{The Hecke algebra is infinitely generated}
\label{sec::infinite_gen}
The goal of this section is to prove that the Hecke algebra $C_{c}(\G,K)$ endowed with the convolution product is infinitely generated.

By contraposition suppose that $C_{c}(\G,K)$ would be a finitely generated algebra with respect to the convolution product. Let $S=\{f_1, \cdots, f_m\}$ be a finite set of generators for $C_{c}(\G,K)$. Then every $f \in C_{c}(\G,K)$ would be written as a finite linear combination of $\{f_1, \cdots, f_m\}$ and finite convolution products of those. By equation~(\ref{equ::gen_decomposition}) we decompose every $f_i \in S$ as $f_i= \sum_{j=1}^{n_i} a_{ij} \mathbbm{1}_{K\gamma_{ij} K}$, where $\gamma_{ij} \in A^{+}$ and $a_{ij} \in \mathbb{C}$. One remarks that the set $S':=\{\mathbbm{1}_{K\gamma_{ij} K}\}_{ij} \subset C_{c}(\G,K)$ also finitely generates $C_{c}(\G,K)$. In addition, by letting $N:=\max_{\gamma_{ij}}(\vert \gamma_{ij} \vert) < \infty$ the set $S'':=\{\mathbbm{1}_{K\gamma K} \; \vert \;  \gamma \in A^{+} \text{ with } \vert \gamma \vert \leq N \}$  finitely generates $C_{c}(\G,K)$ too. Therefore, if $C_{c}(\G,K)$ was finitely generated, we could assume, without loss of generality, that there exists a set $S''=\{\mathbbm{1}_{K\gamma K} \; \vert \;  \gamma \in A^{+} \text{ with } \vert \gamma \vert \leq N \}$ that finitely generates  $C_{c}(\G,K)$.

It remains to investigate the subsets $\{\gamma_1, \cdots, \gamma_n\} \subset A^+$ such that $\{\mathbbm{1}_{K\gamma_1 K}, \cdots, \mathbbm{1}_{K\gamma_n K} \}$ would finitely generate $C_{c}(\G,K)$ with respect to the convolution product. For this we need to solve in $C_{c}(\G,K)$ systems of equations of convolution products and to introduce a notion of linear independence for such systems of equations.

\subsection{Sub-bases and weakly linearly independent equations of degree $n$}
In Section~\ref{subset::product_two_double_classes} we have computed the convolution product of two functions $\mathbbm{1}_{K\gamma_1K}, \mathbbm{1}_{K\gamma_2K} \in C_{c}(\G,K)$, where $\gamma_1, \gamma_2 \in A^+$. We have obtained the equation 
\begin{equation}
\label{emu::conv_two_product_system}
\mathbbm{1}_{K\gamma_1K} \ast \mathbbm{1}_{K\gamma_2K} = \sum\limits_{\vert \eta_{i} \vert < \vert \gamma_1\vert +\vert \gamma_2 \vert, \  i \in I } a_i \mathbbm{1}_{K\eta_i K}+ \sum\limits_{\vert \eta_j \vert = \vert \gamma_1\vert +\vert \gamma_2 \vert, \  j \in J } \mu(K) \mathbbm{1}_{K\eta_jK},
\end{equation}

where $I$ is a finite set, $a_i \in \mathbb{C}$, $\vert J \vert= \kk$, $\eta_i, \eta_j \in A^+$ and $\eta_j$, with $j \in J$, $\vert J \vert =\kk$, are as in Proposition~\ref{prop::KKK}.

\medskip
From now on we consider that the left Haar measure $\mu$ of $\G$ is normalized with $\mu(K)=1$.

\medskip
More generally, by considering $\gamma_1, \gamma_2, \cdots, \gamma_m \in A^+$ we obtain a more general equation 

\begin{equation}
\label{emu::conv_more_product_system}
\mathbbm{1}_{K\gamma_1K} \ast \mathbbm{1}_{K\gamma_2K} \ast \cdots \ast \mathbbm{1}_{K\gamma_mK}= \sum\limits_{\vert \eta_i \vert < \vert \gamma_1\vert + \cdots +\vert \gamma_m \vert, \  i \in I } a_i \mathbbm{1}_{K\eta_i K}+ \sum\limits_{\vert \eta_j \vert = \vert \gamma_1\vert + \cdots +\vert \gamma_m \vert, \  j \in J } \mathbbm{1}_{K\eta_jK},
\end{equation}
 where $I$ is a finite set, $a_i \in \mathbb{C}$, $\vert J \vert= \kk^{m-1}$, $\eta_i, \eta_j \in A^+$ and $\eta_j$ are accordingly as in Proposition~\ref{prop::KKK}.
 
\begin{definition}
\label{def::degree_equation}
Let $f \in C_{c}(\G,K)$. Recall that by equality~(\ref{equ::gen_decomposition})
\begin{equation*}
f= \sum\limits_{i \in I} a_i \mathbbm{1}_{K\eta_iK},
\end{equation*} where $I$ is a finite set, $a_i \in \mathbb{C}$ and $\eta_i \in A^{+}$ for every $i \in I$. We say that (the equation) $f$ is of \textbf{degree} $n$ if for every $i \in I$, $\vert \eta_i \vert \leq n$ and there exists $i \in I$ such that $\vert \eta_i \vert =n$. In particular, equation~(\ref{emu::conv_more_product_system}) is of degree $n$ if $ \vert \gamma_1\vert + \cdots +\vert \gamma_m \vert=n$. 

\end{definition}

\begin{definition}
\label{def::system_equations_degree_n}
Let $I$ be finite. We say that $\{E_i\}_{i \in I}$ is a \textbf{system of equations of degree $n$} if: 
\begin{enumerate}
\item
for every $i \in I$, $E_i:= \mathbbm{1}_{K\gamma_{i_1}K} \ast \mathbbm{1}_{K\gamma_{i_2}K} \ast \cdots \ast \mathbbm{1}_{K\gamma_{i_{m_i}}K}$  for some $\gamma_{i_j} \in A^+$, where $m_i$ is finite and $j \in \{1, \cdots, m_i\}$
\item
$E_i$ is of degree $n$ for every $i \in I$.
\end{enumerate}
\end{definition}

\begin{notation}
\label{not::E_conv}
To simplify the notation, from now on we reserve the letter $E$ to represent  a convolution product of the form $\mathbbm{1}_{K\gamma_1K} \ast \mathbbm{1}_{K\gamma_2K} \ast \cdots \ast \mathbbm{1}_{K\gamma_mK}$, where $\gamma_i \in A^+$ for every $i \in \{1, \cdots, m\}$.\end{notation}

\begin{definition}
\label{def::in_dependent_system_equations_degree_n}
A system of equations $\{E_1, \cdots, E_r\}$ of degree $n$ is \textbf{weakly linearly dependent} if there exist $b_1, \cdots, b_r \in \mathbb{C}$ not all zero such that $b_1 E_1+\cdots + b_r E_r$ is of degree strictly less than $n$. We say that a system of equations $\{E_1, \cdots, E_r\}$ of degree $n$ is \textbf{weakly linearly independent} if it is not weakly linearly dependent.
\end{definition}

\begin{definition}
\label{def::base_degree_n}
We say that the set $\{\mathbbm{1}_{K\gamma_i K}\}_{i \in I}$, where $I$ is finite and $\{\gamma_i\}_{i \in I} \subset A^+$, forms \textbf{a sub-base of $C_{c}(\G,K)$ of degree $ \leq n$} if the following conditions are satisfied:
\begin{enumerate}

\item $\vert \gamma_i \vert \leq n$, for every $i \in I$
\item every function $\mathbbm{1}_{K\eta K}$, with $\eta \in A^+$ and $0 < \vert \eta \vert \leq n $, can be written as a sum of a finite linear combination of $\{\mathbbm{1}_{K\gamma_i K}\}_{i \in I}$ and a finite linear combination of convolution products of $\{\mathbbm{1}_{K\gamma_i K}\}_{i \in I}$ having degree $\leq n$
\item
$\vert I \vert$ is minimal with the above properties.
\end{enumerate}
\end{definition}

\begin{remark}
\label{rem::sub_base_degree_2}
Let $F$ be primitive. It is immediate that a sub-base of $C_{c}(\G,K)$ of degree $2$ exists and it is unique. Its cardinality is $\kk$.
\end{remark}

Before considering the general case, let us warm up proving the following lemma.

\begin{lemma}
\label{lem::sub_base_degree_4}
Let $F$ be primitive. Then $C_{c}(\G,K)$ admits a sub-base of degree $\leq 4$. Its cardinality is $\kk+ \kk^2 (\kk-1)$. A sub-base of $C_{c}(\G,K)$ of degree $\leq 4$ is not unique.
\end{lemma}

\begin{proof} For the proof recall Definition~\ref{def::f_function} and Notation~\ref{not::classes_primitive}. Let also $\{\mathbbm{1}_{K\gamma_j K}\}_{j \in \{1, \cdots, \kk \}}$ be the sub-base of $C_{c}(\G,K)$ degree $2$. 

Let $\eta_1, \eta_2 \in \{\gamma^{j_1}_1, \cdots, \gamma^{j_\kk}_{\kk}\}_{j_1, \cdots, j_\kk \in \{1, \cdots, \kk\}}$. The unique way to obtain functions of the from $\mathbbm{1}_{K\eta_1^{i_1} \eta_2 K}$, with $i_1 \in \{1, \cdots, \kk\}$, as maximal $K$--double coset using the sub-base of degree $2$ is by convoluting  $\mathbbm{1}_{K\eta_1 K}$ and $\mathbbm{1}_{K\eta_2 K}$:
\begin{equation}
\label{equ::equ_deg_2}
\mathbbm{1}_{K\eta_1K} \ast \mathbbm{1}_{K\eta_2 K} = \sum\limits_{\vert \eta_i \vert < \vert \gamma_1\vert +\vert \gamma_2 \vert, \  i \in I } a_i \mathbbm{1}_{K\eta_i K}+ \sum\limits_{i=1}^{\kk} \mathbbm{1}_{K\eta_1^{i} \eta_2 K}.
\end{equation}
Following Proposition~\ref{prop::product_of_trans_lenght_two}, in equation~(\ref{equ::equ_deg_2}) appears all functions $\mathbbm{1}_{K\eta_1^{i_1} \eta_2 K}$, with $i_1 \in \{1, \cdots, \kk\}$. It is the only one where they can appear.

In order to find a sub-base of $C_{c}(\G, K)$ of degree $\leq 4$ we have to choose from equation~(\ref{equ::equ_deg_2})  $f'(2, \kk)=\kk-1$ different elements of $A^+$ of translation length $4$. This number is independent of the choices made for the elements $\{\gamma^{j_i}_i\}_{j,i=1}^{\kk}$. By doing this, one of the terms of degree $4$ appearing in equation~(\ref{equ::equ_deg_2}) will be written in terms of the chosen elements. Without loss of generality, we can suppose that we have chosen $\{\mathbbm{1}_{K\eta_{1}^{i}\eta_2K}\}_{i \in \{2, \cdots, \kk \}}$ to form a sub-base of degree $\leq 4$.  Moreover, by the same reasoning and making a choice, we choose the following set $\{\mathbbm{1}_{K\gamma^{j}_i \gamma_l k}\; \vert \; i,l \in \{1, \cdots, \kk\} \text{ and } j \in \{2, \cdots, \kk\}\} \cup \{\mathbbm{1}_{K\gamma_j K}\}_{j \in \{1, \cdots, \kk \}}$. It is immediate to see that this set is a sub-base of $C_{c}(\G,K)$ of degree $\leq 4$. Indeed, for example the function $\mathbbm{1}_{K\eta_{1}^{1}\eta_2K}$ can be written using equation~(\ref{equ::equ_deg_2}) in the chosen sub-base of degree $\leq 4$. No fewer elements than above can be chosen to form such a sub-base. The cardinality of the chosen sub-base of degree $\leq 4$ is $f(2, \kk) + f(1,\kk)= \kk^2 (\kk-1)+ \kk$; this number is independent of the choices made. Notice that the sub-base $\{\mathbbm{1}_{K\gamma^{j}_i \gamma_l k}\; \vert \; i,l \in \{1, \cdots, \kk\} \text{ and } j \in \{2, \cdots, \kk\}\} \cup \{\mathbbm{1}_{K\gamma_j K}\}_{j \in \{1, \cdots, \kk \}}$ is not unique.

\end{proof}
 
\begin{proposition}
\label{prop::sub_base_deg_r}
Let $F$ be primitive and let $\{\mathbbm{1}_{K\gamma_j K}\}_{j \in \{1, \cdots, \kk \}}$ be the sub-base of $C_{c}(\G,K)$ degree $2$. For every $r \geq 3$ there exists a sub-base of $C_{c}(\G,K)$ of degree $\leq 2r$. Its cardinality is $f(1,\kk)+f(2, \kk)+\cdots + f(r, \kk)$ and it is uniquely determined given the sub-base of degree $\leq 2(r-1)$. In particular, if $\kk=1$ then $f(1,\kk)+f(2, \kk)+\cdots + f(r, \kk)=1$, for every $r\geq 1$.
\end{proposition}

\begin{proof}We prove the proposition by induction on $r$. First recall Definition~\ref{def::f_function} and Notation~\ref{not::classes_primitive}.  By Lemma~\ref{lem::sub_base_degree_4} we take $\{\mathbbm{1}_{K\gamma^{j}_i \gamma_l k}\; \vert \; i,l \in \{1, \cdots, \kk\} \text{ and } j \in \{2, \cdots, \kk\}\} \cup \{\mathbbm{1}_{K\gamma_j K}\}_{j \in \{1, \cdots, \kk \}}$ as a sub-base of $C_{c}(\G,K)$ of degree $\leq 4$.

Fix a sequence $\{\eta_i\}_{i \geq 1} \subset \{\gamma^{j_1}_1, \cdots, \gamma^{j_\kk}_{\kk}\}_{j_1, \cdots, j_\kk \in \{1, \cdots, \kk\}}$.

\medskip
Let $r=3$. Using the chosen sub-base of $C_{c}(\G,K)$ of degree $\leq 4$, we form all equations $E$ of degree $2r=6$ where functions of the form $\mathbbm{1}_{K\eta_{1}^{i_1}\eta_2^{i_2} \eta_{3}K}$, with $i_1,  i_{2} \in \{1, \cdots, \kk\}$, can appear (the total number of these functions is $\kk^{2}$). By a combinatorial argument, Lemmas~\ref{lem::combinatorial_formula},~\ref{lem::sub_base_degree_4} and equality~(\ref{emu::conv_more_product_system}) the total number of these equations is exactly $$\sum\limits_{(k_1, \cdots, k_{r-1}) \in \mathit{Sum}(r)}\left( \begin{array}{c} k_1+ \cdots +  k_{r-1}\\ k_1, \cdots, k_{r-1} \end{array} \right) f'(1,k)^{k_1} f'(2,k)^{k_2}\cdots f'(r-1,k)^{k_{r-1}}= \kk^{2}-f'(3, \kk).$$

We claim that all the above $\kk^{2}-f'(3, \kk)$ equations form a system of weakly linearly independent equations of degree $6$. Indeed, notice that all functions $\mathbbm{1}_{K\eta_{1}^{i_1}\eta_2^{i_2}\eta_3K}$, with $i_1, i_2 \in \{1, \cdots, \kk\}$, appear at least once in the above system of equations. In particular, these functions appear from the convolution product $\mathbbm{1}_{K\eta_{1} K} \ast \mathbbm{1}_{K\eta_2 K} \ast \mathbbm{1}_{K\eta_3 K} $. Moreover, the function $\mathbbm{1}_{K\eta_{1}^{1}\eta_2^{1}\eta_3K}$ appears only once, hence, only from the convolution product $\mathbbm{1}_{K\eta_{1} K} \ast \mathbbm{1}_{K\eta_2 K} \ast \mathbbm{1}_{K\eta_3 K} $. This is because the sub-base of degree $\leq 4$ is chosen to be $\{ \mathbbm{1}_{K\gamma^{j}_i \gamma_l K} \; \vert \; i,l \in \{1, \cdots, \kk\} \text{ and } j \in \{2, \cdots, \kk\}\} \cup \{\mathbbm{1}_{K\gamma_j K}\}_{j \in \{1, \cdots, \kk \}}$. Moreover, functions of the form $\mathbbm{1}_{K \eta_{1}^{i} \eta_2^{1} \eta_3 K}, \mathbbm{1}_{K \eta_{1}^{1} \eta_2^{i} \eta_3 K}$, for every $i \in \{2, \cdots, \kk \}$ appear twice: once from the convolution product $\mathbbm{1}_{K\eta_{1} K} \ast \mathbbm{1}_{K\eta_2 K} \ast \mathbbm{1}_{K\eta_3 K}$ and once, respectively, from the convolution product $\mathbbm{1}_{K\eta_{1}^{j}\eta_2 K} \ast \mathbbm{1}_{K\eta_3 K}$, $\mathbbm{1}_{K\eta_{1}K}  \ast \mathbbm{1}_{K\eta_2^{j}\eta_3 K}$, where $j \in \{2, \cdots, \kk \}$. These remarks imply that the system of equations we are interested in are indeed weakly linearly independent. This proves our claim.

By the theory of linear algebra, in order to form a sub-base of degree $\leq 6$, given the sub-base of degree $\leq 4$, we have to choose (from the above system of equations of degree $6$) $f'(3, \kk)$ characteristic functions corresponding to $K$--double cosets of degree $6$. This argument is independent and valid for every of the $\kk^{3}$ systems of equations.

By the proof of the above claim and above facts, to obtain a sub-base of degree $\leq 6$ one \textit{uniquely} can add the set $\{\mathbbm{1}_{K\gamma_{j_1}^{i_1}\gamma_{j_2}^{i_2}\gamma_{j_{3}} K}\}$, where $i_1, i_2 \in \{2, \cdots, \kk \}$ and $j_1, j_2, j_{3} \in \{1, \cdots, \kk \}$. This set together with the sub-base of $C_{c}(\G,K)$ of degree $\leq 4$ is minimal and its cardinality is indeed $f(1,\kk)+f(2, \kk)+f(3, \kk)$. In particular, all the functions of the form $\mathbbm{1}_{K\gamma_{j_1}^{1}\gamma_{j_2}^{i_2}\gamma_{j_{3}} K}, \mathbbm{1}_{K\gamma_{j_1}^{i_1}\gamma_{j_2}^{1}\gamma_{j_{3}} K}$, with $i_1, i_2, j_1, j_2, j_{3} \in \{1, \cdots, \kk \}$, are written using the chosen sub-base of degree $\leq 6$; they do not appear as elements of that sub-base.

Let us now suppose that the conclusion of the proposition is true for all $\leq r$ and that the sub-base of $C_{c}(\G,K)$ of degree $\leq 2r$ does not contain any function of the form $$\mathbbm{1}_{K\gamma_{j_1}^{i_1}\gamma_{j_2}^{i_2} \cdots \gamma_{j_{l-1}}^{i_{l-1}}\gamma_{j_l} K}$$ where $l \leq r$ and at least one of $\{i_1, \cdots, i_{l-1}\}$ is $1$. We have to prove this is also true for $r+1$. Indeed, using only the sub-base of degree $\leq 2r$, constructed in the previous induction steps, we form all equations $E$ of degree $2(r+1)$ where functions of the form $\mathbbm{1}_{K\eta_{1}^{i_1}\eta_2^{i_2} \cdots\eta_r^{i_r} \eta_{r+1}K}$, with $i_1, i_2, \cdots, i_r \in \{1, \cdots, \kk\}$, can  appear (the total number of these functions is $\kk^{r}$). By a combinatorial argument, Lemmas~\ref{lem::combinatorial_formula},~\ref{lem::sub_base_degree_4} and equality~(\ref{emu::conv_more_product_system}) the total number of these equations is exactly $$\sum\limits_{(k_1, \cdots, k_{r}) \in \mathit{Sum}(r+1)}\left( \begin{array}{c} k_1+ \cdots +  k_{r}\\ k_1, \cdots, k_{r} \end{array} \right) f'(1,k)^{k_1} f'(2,k)^{k_2}\cdots f'(r,k)^{k_{r}}= \kk^{r}-f'(r+1, \kk).$$

We claim that all the above $\kk^{r}-f'(r+1, \kk)$ equations form a system of weakly linearly independent equations of degree $2(r+1)$. Indeed, every function $\mathbbm{1}_{K\eta_{1}^{i_1}\eta_2^{i_2} \cdots \eta_r^{i_r} \eta_{r+1}K}$ with $i_1, \cdots, i_{r} \in \{1, \cdots, \kk\}$ appears at least once in the above system of equations, specifically,  from the convolution product $\mathbbm{1}_{K\eta_{1} K} \ast \mathbbm{1}_{K\eta_2 K} \ast  \cdots \ast \mathbbm{1}_{K\eta_{r+1} K} $. In addition, the function $\mathbbm{1}_{K\eta_{1}^{1}\eta_2^{1} \cdots \eta_r^{1} \eta_{r+1}K}$ appears exactly once in that system of equations. Moreover, every equation of the above system determines in a unique way a function $\mathbbm{1}_{K\eta_{1}^{i_1}\eta_2^{i_2} \cdots \eta_r^{i_r} \eta_{r+1}K}$ where the number of appearances of $1$ among the coefficients $i_1, \cdots, i_r$ is  maximal. This function uniquely depends on the convolution product involved in that equation. In addition, there exists at least one different equation of that system of equations such that the latter mention function appears, but it is not anymore `maximal'. The only exception is the function $\mathbbm{1}_{K\eta_{1}^{1}\eta_2^{1} \cdots \eta_r^{1} \eta_{r+1}K}$.

If the above system of equations $(E_1, \cdots, E_n)$, where $n= \kk^{r}-f'(r+1, \kk)$, was not weakly linearly independent there would exist coefficients $c_1, \cdots, c_n \in \mathbb{C}$, not all zero, such that $c_1E_1+ \cdots +c_n E_n$ is of degree strictly less then $2(r+1)$. By the remarks made above this would imply that we have reduced all functions $\mathbbm{1}_{K\eta_{1}^{i_1}\eta_2^{i_2} \cdots \eta_r^{i_r} \eta_{r+1}K}$ where at least one of $\{i_1, \cdots, i_l\}$ is $1$. This cannot be possible as every equation determines in a unique way a `maximal' such function. Therefore, the weakly linearly independence follows. This argument is independent and valid for every of the other $\kk^{r+1}$ systems of equations. 

It remains to prove that there is only one choice in order to complete the sub-base of degree $\leq 2r$ to a sub-base of degree $ \leq 2(r+1)$ and that the chosen sub-base of degree $2(r+1)$ is indeed $\mathbbm{1}_{K\eta_{1}^{i_1}\eta_2^{i_2} \cdots \eta_r^{i_r} \eta_{r+1}K}$, for $i_1, \cdots, i_{r} \in \{2, \cdots, \kk\}$. This follows from the fact that every function of the form $\mathbbm{1}_{K\eta_{1}^{i_1}\eta_2^{i_2} \cdots \eta_r^{i_r} \eta_{r+1}K}$, where $i_1, \cdots, i_{r} \in \{2, \cdots, \kk\}$ appears more than every other function $\mathbbm{1}_{K\eta_{1}^{i_1}\eta_2^{i_2} \cdots \eta_r^{i_r} \eta_{r+1}K}$, where at least one of the coefficients $i_1, \cdots, i_r$ is $1$. This concludes the induction step and also the proof of the proposition.

\end{proof}

\subsection{The proof of the main theorem}
\label{subsec::infinite_generation}

\begin{theorem}
\label{thm::inf_gen_hecke_alg}
Let $F$ be primitive. If $\kk=1$ then the Hecke algebra $C_{c}(\G,K)$ is finitely generated admitting only one generator. If $\kk >1$ then the Hecke algebra $C_{c}(\G,K)$ is infinitely generated with an infinite presentation.
\end{theorem}

\begin{proof}Let $\kk=1$. Then $F$ is $2$--transitive. By Remark~\ref{rem::K_cosets} we have that $A^{+}=<a>^{+}$ where $a$ is a hyperbolic element of $\G$, with $\vert a \vert=2$ and $x \in \Min(a)$. The $K$--double cosets of $\G$ are just $\{Ka^{n}K\}_{n \geq 0}$. It is then easy to see that the function $\mathbbm{1}_{K a K}$ alone generates the Hecke algebra $C_{c}(\G,K)$ (one can also apply the general result of Proposition~\ref{prop::sub_base_deg_r}). 

Consider now the case $\kk>1$ and suppose that $C_{c}(\G,K)$ is a finitely generated algebra with respect to the convolution product. Arguing as in the introduction to Section~\ref{sec::infinite_gen} we can assume, without loss of generality,  that $C_{c}(\G,K)$ is finitely generated by the set $S''=\{\mathbbm{1}_{K\gamma K} \; \vert \;  \gamma \in A^{+} \text{ with } \vert \gamma \vert \leq N \}$, for some $N \in \mathbb{N}^*$. By Proposition~\ref{prop::sub_base_deg_r}, there exists a sub-base of $C_{c}(\G,K)$ of degree $\leq N$ and by our assumption this sub-base of $C_{c}(\G,K)$ of degree $\leq N$ finitely generates $C_{c}(\G,K)$. This is in contradiction with Proposition~\ref{prop::sub_base_deg_r} as the sub-base of $C_{c}(\G,K)$ of degree $\leq N+1$ strictly contains the one of degree $\leq N$. Therefore, $C_{c}(\G,K)$ is infinitely generated. Its infinitely presentation is the one coming from the convolution products of the generators.

 \end{proof}


\end{document}